\documentclass[a4paper,reqno,10pt]{amsart}

\usepackage{hyperref}
\usepackage{a4wide}

\usepackage{amssymb}
\usepackage{amstext}
\usepackage{amsmath}
\usepackage{amscd}
\usepackage{amsthm}
\usepackage{amsfonts}

\usepackage{graphicx}
\usepackage{latexsym}
\usepackage{mathabx}
\usepackage{mathrsfs}

\usepackage{enumitem}
\setlist[enumerate,1]{label={\upshape(\arabic*)}}
\setlist[enumerate,2]{label={\upshape(\alph*)}}

\usepackage{tikz}
\usetikzlibrary{cd,arrows,matrix,backgrounds,shapes,positioning,calc,decorations.markings,decorations.pathmorphing,decorations.pathreplacing}
\tikzset{blackv/.style={circle,fill=black,inner sep=3pt,outer sep=3pt},
         whitev/.style={circle,fill=white,draw=black,inner sep=3pt,outer sep=3pt},
         blabel/.style={circle,draw=black,inner sep=1.5pt,outer sep=0pt},
         redv/.style={circle,fill=red,inner sep=3pt,outer sep=3pt},
         block/.style={draw,rectangle split,rectangle split horizontal,rectangle split parts=#1},
         symbol/.style={
           draw=none,
           every to/.append style={
             edge node={node [sloped, allow upside down, auto=false]{$#1$}}}}
}

\usepackage{array}
\newcolumntype{C}{>{$}c<{$}}

\usepackage{makecell}

\newtheorem{theorem}{Theorem}[section]
\newtheorem{theoremi}{Theorem}

\newtheorem{propositioni}[theoremi]{Proposition}

\newtheorem{corollary}[theorem]{Corollary}

\newtheorem*{lemma*}{Lemma}
\newtheorem*{theorem*}{Theorem}
\newtheorem{proposition}[theorem]{Proposition}
\newtheorem{definition-proposition}[theorem]{Definition-Proposition}

\theoremstyle{definition}
\newtheorem{definition}[theorem]{Definition}

\newtheorem{remark}[theorem]{Remark}
\newtheorem{example}[theorem]{Example}

\newtheorem*{ack}{Acknowledgements}
\newtheorem*{conv}{Conventions and notation}

\newcommand{\Ga}{\Gamma}
\newcommand{\de}{\delta}

\newcommand{\dra}{\dashrightarrow}
\newcommand{\ovs}{\overset}
\newcommand{\ci}{\circ}
\newcommand{\co}{\colon}
\newcommand{\sha}{\sharp}
\newcommand{\Si}{\Sigma}

\newcommand{\sfra}{{\mathfrak{s}}}

\newcommand{\Acal}{\mathcal{A}}

\newcommand{\Ccal}{\mathcal{C}}
\newcommand{\Dcal}{\mathcal{D}}

\newcommand{\Pcal}{\mathcal{P}}

\newcommand{\Scal}{\mathcal{S}}

\newcommand{\Ebb}{\mathbb{E}}
\newcommand{\Fbb}{\mathbb{F}}

\newcommand{\Zbb}{\mathbb{Z}}

\newcommand{\op}{\operatorname{op}\nolimits}
\newcommand{\ex}{\operatorname{ex}\nolimits}
\newcommand{\RHom}{\mathbf{R}\strut\kern-.2em\operatorname{Hom}\nolimits}

\newcommand{\Image}{\operatorname{Im}\nolimits}
\newcommand{\Kernel}{\operatorname{Ker}\nolimits}
\newcommand{\Cokernel}{\operatorname{Coker}\nolimits}

\newcommand{\Ab}{\mathcal{A}b}

\renewcommand{\ker}{\Kernel}
\newcommand{\un}{\underline}

\newcommand{\se}{\subseteq}
\newcommand{\opl}{\oplus}

\DeclareMathOperator{\moduleCategory}{\mathsf{mod}} \renewcommand{\mod}{\moduleCategory}
\DeclareMathOperator{\Mod}{\mathsf{Mod}}
\DeclareMathOperator{\defe}{\mathsf{def}}

\DeclareMathOperator{\Hex}{\mathsf{Hex}}
\DeclareMathOperator{\Rel}{\mathsf{Rel}}
\DeclareMathOperator{\PFun}{\mathsf{PFun}}

\DeclareMathOperator{\Serre}{\mathsf{Serre}}

\DeclareMathOperator{\coh}{\mathsf{coh}}

\DeclareMathOperator{\id}{\mathsf{id}}
\DeclareMathOperator{\Id}{\mathsf{Id}}

\newcommand{\iso}{\cong}

\newcommand{\defl}{\twoheadrightarrow}

\newcommand{\sst}[1]{\substack{#1}}
\newcommand{\wch}{\widecheck}
\newcommand{\wha}{\widehat}
\newcommand{\wti}{\widetilde}

\newcommand{\Ecal}{\mathcal{E}}

\DeclareMathOperator{\CEs}{(\mathcal{C},\mathbb{E},\mathfrak{s})}

\numberwithin{equation}{section}

\makeatletter
\@namedef{subjclassname@2020}{\textup{2020} Mathematics Subject Classification}
\makeatother

\begin{document}
\title[Relative extriangulated categories arising from half exact functors]{Relative extriangulated categories arising from half exact functors}

\author[A. Sakai]{Arashi Sakai}
\address{Graduate School of Mathematics, Nagoya University, Chikusa-ku, Nagoya. 464-8602, Japan}
\email{m20019b@math.nagoya-u.ac.jp}
\begin{abstract}
Relative theories(=closed subfunctors) are considered in exact, triangulated and extriangulated categories by Dr\"{a}xler-Reiten-Smal\o-Solberg-Keller, Beligiannis and Herschend-Liu-Nakaoka, respectively. We give a construction method of closed subfunctors from given half exact functors which contains existing constructions. Moreover, if an extriangulated category has enough projective objects, then every closed subfunctor is obtained by this construction.
\end{abstract}

\maketitle

\tableofcontents

\section{Introduction}\label{sec:1}

Recently, Nakaoka and Palu have introduced the notion of \emph{extriangulated categories} which unifies abelian categories, exact categories and triangulated categories \cite{NP}. An extriangulated category consists of a triple $\CEs$ with an additive category $\Ccal$, a biadditive functor $\Ebb\co\Ccal^{\op}\times\Ccal\to\Ab$ and a realization $\sfra$ which associate elements of $\Ebb(C,A)$ to an equivalence class of sequences of the form $[A\to B\to C]$. There are many studies related to extriangulated categories: Auslander-Reiten theory \cite{INP}, sub-bifunctors of $\Ebb$ \cite{BBGH}, localizations \cite{NOS}, applications to extension-closed subcategories of triangulated categories \cite{Pre} \cite{Jin}, and so on.

The concept of \emph{relative theories} has been initiated in exact categories \cite{DRSSK} and generalized in $n$-exangulated categories \cite{HLN} which contains extriangulated categories in the case of $n=1$. Relative theories are explicated by use of \emph{closed subfunctors}. An additive subfunctor $\Fbb$ of $\Ebb$ is a closed subfunctor if $(\Ccal,\Fbb,\sfra|_\Fbb)$ is an extriangulated category where $\sfra|_\Fbb$ is a restriction of $\sfra$ to $\Fbb$. The notion of \emph{proper classes in extriangulated categories} is introduced by Hu-Zhang-Zhou \cite{HZZ} as a generalization of proper classes of triangles introduced by Beligiannis \cite{Bel}. We show that it is the same notion as relative theories in Appendix A. 

In this paper, we give a new way to construct a closed subfunctor of an extriangulated category. The key notion is a \emph{half exact functor} from an extriangulated category to an abelian category, which coincides with a usual half exact functor and a homological functor in the case we consider exact categories and triangulated categories, respectively. The following statement is a starting point in this paper. 

\begin{propositioni}[= Proposition \ref{prop:left}] \label{prop:a}
Let $\CEs$ be an extriangulated category and $H\colon\Ccal\to\Acal$ a half exact functor to an abelian category $\Acal$. 
Then there exists a unique maximal closed subfunctor $\Fbb$ of $\Ebb$ such that $H$ becomes a right exact functor from $(\Ccal, \Fbb, \sfra|_{\Fbb})$ to $\Acal$. 
\end{propositioni}

This construction contains several existing constructions, see Examples~\ref{ex:projectivize} and \ref{ex:proper}. In \cite{Eno}, closed subfunctors are classified by Serre subcategories of the category of \emph{defects of $\Ebb$-extensions}. We deal with a relation between the above proposition and this classification result via \emph{projectivization functors} defined in Definition~\ref{def:projectivize}. In Proposition~\ref{prop:fer}, we show that every half exact functor vanishing on projective objects defines the same closed subfunctor as the one obtained from a projectivization functor. The following is our main theorem.

\begin{theoremi}[= Theorem \ref{thm:tri}] \label{thm:b}
Let $(\Ccal, \Ebb, \sfra)$ be an extriangulated category which has enough projective objects.
Then there exists a bijective correspondence between the followings.
\begin{enumerate}
  \item Projectivization functors. 
  \item Serre subcategories of the category of defects $\defe\Ebb$.
  \item Closed subfunctors of $\Ebb$.
\end{enumerate}
\end{theoremi}

In particular, every closed subfunctor is obtained from a certain half exact functor through this correspondence (Corollary~\ref{cor:surj}).

\begin{conv}
We always assume that \emph{all subcategories are full and closed under isomorphisms}. Throughout this paper, \emph{$\Ccal$ denotes a skeletally small additive category}. We denote by $\Mod\Ccal$ the category of all right $\Ccal$-modules, that is, contravariant additive functors $\Ccal^{\op}\to\Ab$.
\end{conv}

\section{Preliminaries}\label{sec2}

We collect some basic definitions and facts about extriangulated categories. For more details, see \cite{NP}.

\begin{definition}\label{def:ext}
Let $\Ccal$ be an additive category equipped with a biadditive functor $\Ebb\colon\Ccal^{\op}\times\Ccal\to\Ab$. 
\begin{itemize}
  \item We call an element $\de\in\Ebb(C, A)$ an \emph{$\Ebb$-extension}. For any morphisms $f\colon C'\to C$ and $g\colon A\to A'$ in $\Ccal$, we denote $\Ebb(f, A)(\de)$ and $\Ebb(C, g)(\de)$ by $\de f$ and $ g \de$, respectively. Since $\Ebb$ is a bifunctor, we have $(g\de)f=g(\de f)$, and denote it by $g\de f$. 
  \item A pair $(a,c)$ of morphisms with $a\co A\to A'$ and $c\co C'\to C$ is a \emph{morphism of $\Ebb$-extensions} from $\de\in\Ebb(C,A)$ to $\de'\in\Ebb(C',A')$ if we have $a\de=\de'c$. We denote it by $(a,c)\co\de\to\de'$.
  \item We denote by $\de\opl\de'$ the extension in $\Ebb(C\opl C',A\opl A')$ which corresponds to $(\de,0,0,\de')\in\Ebb(C,A)\opl\Ebb(C,A')\opl\Ebb(C',A)\opl\Ebb(C',A')$ via the natural isomorphism. 
  \item We call $0\in\Ebb(C,A)$ the \emph{split $\Ebb$-extension}.
  \item For any morphism $a\co A\to A'$, we denote a morphism $\Ebb(-,a)\co\Ebb(-,A)\to\Ebb(-,A')$ in $\Mod\Ccal$ by $a_{\ast}$. Dually, we denote a morphism $\Ebb(a,-)\co\Ebb(A',-)\to\Ebb(A,-)$ in $\Mod\Ccal^{\op}$ by $a^{\ast}$. 
  \item Let $\de\in\Ebb(C,A)$ be an $\Ebb$-extension. By the Yoneda lemma, there exists a morphism $\Ccal(-,C)\to\Ebb(-,A)$ in $\Mod\Ccal$, denoted by $\de_{\sha}$. Dually, there exists a morphism $\Ccal(A,-)\to\Ebb(C,-)$ in $\Mod\Ccal^{\op}$, denoted by $\de^{\sha}$. 
\end{itemize}
\end{definition}

\begin{definition}
Let $A$ and $C$ be objects in $\Ccal$. Sequences of morphisms $A\ovs{f}{\to}B\ovs{g}{\to}C$ and $A\ovs{f'}{\to}B'\ovs{g'}{\to}C$ are \emph{equivalent} if there exists an isomorphism $b\co B\to B'$ which makes the following diagram commute.
  \begin{equation}
  \begin{tikzcd}
    A \rar["f"] \dar[equal] & B \rar["g"] \dar["b"] & C \dar[equal] \\
    A \rar["f'"'] & B' \rar["g'"'] & C
  \end{tikzcd}
  \end{equation}
We denote by $[A\ovs{f}{\to}B\ovs{g}{\to}C]$ the equivalence class to which $A\ovs{f}{\to}B\ovs{g}{\to}C$ belongs.  
\end{definition}

\begin{definition}
Let $\sfra$ be a map from $\Ebb(C,A)$ to the set of equivalence classes of sequences of the form $[A\to B\to C]$ for any $A,C\in\Ccal$.
\begin{itemize}
  \item We call $[A\ovs{f}{\to} B\ovs{g}{\to} C]$ a \emph{realization} of $\de\in\Ebb(C, A)$ if we have $\sfra(\de)=[A\ovs{f}{\to} B\ovs{g}{\to} C]$.
  \item We call a morphism $f\co A\to B$ an \emph{$\sfra$-inflation} if there is an $\Ebb$-extension $\de\in\Ebb(C,A)$ such that $\sfra(\de)=[A\ovs{f}{\to} B\to C]$. Dually, we call a morphism $g\co B\to C$ an \emph{$\sfra$-deflation} if there is an $\Ebb$-extension $\de\in\Ebb(C,A)$ such that $\sfra(\de)=[A{\to} B\ovs{g}{\to} C]$.
  \item We denote by $A\ovs{f}{\to} B\ovs{g}{\to} C\ovs{\de}{\dra}$ a pair of $\de\in\Ebb(C,A)$ and a sequence $A\ovs{f}{\to}B\ovs{g}{\to}C$ which satisfies $\sfra(\de)=[A\ovs{f}{\to} B\ovs{g}{\to} C]$, and call it an \emph{$\sfra$-triangle}.
  \item We call a triple of morphisms $(a,b,c)$ a \emph{morphism of $\sfra$-triangles} if it satisfies $a\de=\de'c$ and makes the following diagram commutative. 
  \begin{equation}\label{morphtri}
  \begin{tikzcd}
    A \rar["f"] \dar["a"'] & B \rar["g"] \dar["b"] & C \rar[dashed, "\de"]\dar["c"] & \  \\
    A' \rar["f'"'] & B' \rar["g'"'] & C' \rar[dashed, "\de'"] & \ 
  \end{tikzcd}
  \end{equation}
\end{itemize}  
\end{definition}

\begin{definition}\label{def:real}
We say $\sfra$ is an \emph{additive realization of $\Ebb$} if it satisfies the following conditions.
\begin{itemize}
\item A morphism of $\Ebb$-extensions induces a morphism of $\sfra$-triangles, that is, for any morphism $(a,c)\co\de\to\de'$ with $\de\in\Ebb(C,A), \de'\in\Ebb(C',A')$, we have a morphism of $\sfra$-triangles as $(\ref{morphtri})$.
\item For any split $\Ebb$-extension $0\in\Ebb(C,A)$, we have $\sfra(0)=[A\ovs{\bigl(\substack{1 \\ 0}\bigr)}{\to}A\opl C\ovs{(01)}{\to} C]$. 
\item Let $\de\in\Ebb(C,A)$ and $\de'\in\Ebb(C',A')$ be any $\Ebb$-extensions. For any $\sfra$-triangles $A\ovs{f}{\to} B\ovs{g}{\to} C\ovs{\de}{\dra}$ and $A'\ovs{f}{\to} B'\ovs{g'}{\to} C'\ovs{\de'}{\dra}$, we have 
  \[
   \sfra(\de\opl\de')=[A\opl A'\ovs{f\opl f'}{\to} B\opl B'\ovs{g\opl g'}{\to} C\opl C'].
  \]
\end{itemize} 
\end{definition}

\begin{definition}$($\cite[Definition 2.12]{NP}$)$
An \emph{extriangulated category} is a triplet $\CEs$ where
\begin{itemize}
  \item $\Ccal$ is an additive category
  \item $\Ebb\colon\Ccal^{\op}\times\Ccal\to\Ab$ is a biadditive functor
  \item $\sfra$ is an additive realization
\end{itemize}
satisfying the following conditions {\rm (ET3), (ET4)} and their duals ${\rm (ET3)}^{\op}, {\rm (ET4)}^{\op}$. 
\begin{itemize}
  \item[{\rm (ET3)}]
  For any diagram
  \[
  \begin{tikzcd}
    A \rar["f"] \dar["a"'] & B \rar["g"] \dar["b"] & C \rar[dashed, "\de"] & \  \\
    A' \rar["f'"'] & B' \rar["g'"'] & C' \rar[dashed, "\de'"] & \ 
  \end{tikzcd}
  \]
  where both rows are $\sfra$-triangles, and satisfying $bf=f'a$, there exists a morphism $c\co C\to C'$ such that $(a,b,c)$ gives a morphism of $\sfra$-triangles as $(\ref{morphtri})$.
  \item[{\rm (ET4)}]
  Let $A\ovs{f}{\to} B\ovs{f'}{\to} D\ovs{\de}{\dra}$ and $B\ovs{g}{\to} C\ovs{g'}{\to} F\ovs{\rho}{\dra}$ be $\sfra$-triangles.  Then there exists a commutative diagram
  \[
  \begin{tikzcd}
    A \rar["f"] \dar[equal] & B \rar["f'"] \dar["g"] & D \rar[dashed, "\de"]\dar["d"] & \  \\
    A \rar["gf"] & C \rar["h'"]\dar["g'"] & E \rar[dashed, "\de'"]\dar["e"] & \ \\
     \ & F\rar[equal]\dar[dashed,"\rho"] & F\dar[dashed,"f'\rho"] & \ \\
     \ & \ & \ & \ 
  \end{tikzcd}
  \]
where $A\ovs{gf}{\to} C\ovs{h'}{\to} E\ovs{\de'}{\dra}$ and $D\ovs{d}{\to} E\ovs{e}{\to} F\ovs{f'\rho}{\dra}$ are $\sfra$-triangles and $\de'd=\de$ and $\rho e=f\de'$ hold.  
\end{itemize}
\end{definition}

The followings are basic properties of extriangulated categories.

\begin{proposition}$($\cite[Corollary 3.12]{NP}\label{prop:exactseq}$)$
Let $\CEs$ be an extriangulated category. 
\begin{enumerate}
  \item For any $\sfra$-triangle $A\ovs{f}{\to} B\ovs{g}{\to} C\ovs{\de}{\dra}$, we obtain an exact sequence
    \[
    \begin{tikzcd}
       \Ccal(-,A)\rar["f\ci-"]& \Ccal(-,B) \rar["g\ci-"] & \Ccal(-,C) \rar["\de_{\sha}"] & \Ebb(-,A)\rar["f_{\ast}"] & \Ebb(-,B)\rar["g_{\ast}"] & \Ebb(-,C)
    \end{tikzcd}
    \]
    in $\Mod\Ccal$.
  \item For any $\sfra$-triangle $A\ovs{f}{\to} B\ovs{g}{\to} C\ovs{\de}{\dra}$, we obtain an exact sequence 
    \[
    \begin{tikzcd}
       \Ccal(C,-)\rar["-\ci g"]& \Ccal(B,-) \rar["-\ci f"] & \Ccal(A,-) \rar["\de^{\sha}"] & \Ebb(C,-)\rar["g^{\ast}"] & \Ebb(B,-)\rar["f^{\ast}"] & \Ebb(A,-)
    \end{tikzcd}
    \]
    in $\Mod\Ccal^{\op}$.
\end{enumerate}
\end{proposition}

\begin{proposition}$($\cite[Proposition 1.20]{LN}\label{prop:bicar}$)$
Let $A\ovs{f}{\to} B\ovs{g}{\to} C\ovs{\de}{\dra}$ be an $\sfra$-triangle and $a\colon A\to A'$ a morphism in $\Ccal$. 
For any $\sfra$-triangle $A'\ovs{f'}{\to} B'\ovs{g'}{\to} C\ovs{a\de}{\dra}$, there exists a morphism of $\sfra$-triangles
  \[
  \begin{tikzcd}
    A \rar["f"] \dar["a"'] \ar[rd, phantom, "{\rm (*)}"] & B \rar["g"] \dar["b"] & C \rar[dashed, "\de"]\dar[equal] & \  \\
    A' \rar["f'"'] & B' \rar["g'"] & C \rar[dashed, "a\de"] & \ 
  \end{tikzcd}
  \]
such that
  \[
  \begin{tikzcd}
    A \rar["\bigl(\substack{a \\ f}\bigr)"] & A'\opl B \rar["(-f' b)"] & B' \rar[dashed, "\de g'"]& \ 
  \end{tikzcd}
  \]
is an $\sfra$-triangle.  
In particular, $(*)$ is a weak pullback and weak pushout diagram.       
\end{proposition}

The notion of relative theories in exact categories is introduced in \cite{DRSSK}, and generalized in the case of extriangulated categories in \cite{HLN}.

\begin{definition}$($\cite[Definition 3.8, Claim 3.9]{HLN}$)$
Let $\Fbb$ be an additive subfunctor of $\Ebb$. We denote the restriction of $\sfra$ to $\Fbb$ by $\sfra|_{\Fbb}$. Then $(\Ccal, \Fbb, \sfra|_{\Fbb})$ automatically satisfies the definition of an extriangulated category without {\rm (ET4)} and ${\rm (ET4)}^{\op}$.
\end{definition}

\begin{proposition}$($\cite[Proposition 3.16]{HLN}\label{prop:closed}$)$
The following conditions are equivalent for an additive subfunctor $\Fbb$ of $\Ebb$.
\begin{enumerate}
  \item $(\Ccal, \Fbb, \sfra|_{\Fbb})$ satisfies {\rm (ET4)} and ${\rm (ET4)}^{\op}$, and becomes an extriangulated category.
  \item $\sfra|_{\Fbb}$-inflations are closed under compositions.
  \item $\sfra|_{\Fbb}$-deflations are closed under compositions.
\end{enumerate}  
\end{proposition}

We call an additive subfunctor $\Fbb$ of $\Ebb$ a \emph{closed subfunctor} if it satisfies the above equivalent conditions. 
In this case, we call $(\Ccal,\Fbb,\sfra|_{\Fbb})$ a \emph{relative theory} of $\CEs$. For a relative theory $(\Ccal,\Fbb,\sfra|_{\Fbb})$, we denote it by $(\Ccal, \Fbb)$, abbreviately.

\section{Closed subfunctors arising from half exact functors}\label{sec3}
In this section, we introduce methods of constructing closed subfunctors of an extriangulated category $\CEs$ from half exact functors. 
First, let us recall the notion of a half exact functor, which is also called a homological functor in \cite{LN}. 

\begin{definition}$($\cite[Definition 2.7]{Oga}, \cite[Definition 3.3]{LN}$)$
Let $\Acal$ be an abelian category. 
An additive functor $H\colon\Ccal\to\Acal$ is a \emph{half exact functor} if the sequence 
$HA\ovs{Hf}{\to}HB\ovs{Hg}{\to}HC$ is exact at $HB$ for any $\sfra$-triangle $A\ovs{f}{\to}B\ovs{g}{\to}C\dra$. 
Moreover, if $Hg$ (resp. $Hf$) is an epimorphism (resp. monomorphism), we call $H$ a \emph{right exact functor} (resp. \emph{left exact functor}). 
\end{definition}

\begin{definition}
Let $H\co\Ccal\to\Acal$ be a half exact functor. We define a subset $\Ebb_{R}^{H}(C, A)$ of $\Ebb(C, A)$ consisting of $\de$ such that for an $\sfra$-triangle $A\ovs{f}{\to}B\ovs{g}{\to}C\ovs{\de}{\dra}$, we have that $Hg$ is an epimorphism in $\Acal$. 
Note that this definition is well-defined, that is, it does not depend on the choice of an $\sfra$-triangle of $\de$.
Similarly, we define a subset $\Ebb_{L}^{H}(C, A)$ of $\Ebb(C, A)$ consisting of $\de$ such that $Hf$ is a monomorphism in $\Acal$. 
\end{definition}

Then $\Ebb_R^H$ defines the maximum closed subfunctor such that $H$ becomes a right exact functor:

\begin{proposition}\label{prop:left}
Let $H\co\Ccal\to\Acal$ be a half exact functor. Then the following assertions hold.
\begin{enumerate}
  \item $\Ebb_{R}^{H}$ is a closed subfunctor of $\Ebb$, hence $(\Ccal, \Ebb_{R}^{H})$ is a relative theory of $\CEs$.
  \item $H$ restricts to a right exact functor $H\colon (\Ccal, \Ebb_{R}^{H})\to\Acal$. 
  \item Let $(\Ccal, \Fbb)$ be a relative theory of $(\Ccal,\Ebb,\sfra)$. 
  If $H$ restricts to a right exact functor $H\colon(\Ccal, \Fbb)\to\Acal$, then we have $\Fbb\subseteq\Ebb_{R}^{H}$.
\end{enumerate}
\end{proposition}

\begin{proof}
(1) First, we show that $\Ebb_{R}^{H}$ is a subfunctor of $\Ebb$. 
For any $\de\in\Ebb_{R}^{H}(C, A)$ and any morphism $a \colon A\to A'$, we show $a\de\in\Ebb_{R}^{H}(C, A')$. By Definition~\ref{def:real}, we obtain the following commutative diagram.
  \[
  \begin{tikzcd}
    A \rar["f"] \dar["a"] & B \rar["g"] \dar["b"] & C \rar[dashed, "\de"]\dar[equal] & \  \\
    A' \rar["f'"] & B' \rar["g'"] & C \rar[dashed, "a\de"] & \ 
  \end{tikzcd}
  \]
It is clear that if $Hg$ is epi, then so is $Hg'$, hence we have $a\de\in\Ebb_{R}^{H}(C,A')$.

For any $\de\in\Ebb_{R}^{H}(C, A)$ and any morphism $c \colon C'\to C$, we show $\de c\in\Ebb_{R}^{H}(C', A)$. 
By the dual of Proposition~\ref{prop:bicar}, we have the following commutative diagram,
  \[
  \begin{tikzcd}
    A \rar["f'"] \dar[equal] & B' \rar["g'"] \dar["b"] & C' \rar[dashed, "\de c"]\dar["c"] & \  \\
    A \rar["f"] & B \rar["g"] & C \rar[dashed, "\de"] & \ 
  \end{tikzcd}
  \] 
satisfying that $B'\ovs{\bigl(\substack{b \\ -g'}\bigr)}{\to}B\opl C'\ovs{(g c)}{\to}C\ovs{f'\de}{\dra}$ is an $\sfra$-triangle. 
Applying $H$ to the above diagram, we obtain a commutative diagram
  \[
  \begin{tikzcd}
    HA \rar["Hf'"] \dar[equal] & HB' \rar["Hg'"] \dar["Hb"'] \ar[rd, phantom, "{\rm (*)}"] & HC' \dar["Hc"] \\
    HA \rar["Hf"'] & HB \rar["Hg"'] & HC 
  \end{tikzcd}
  \]
  in $\Acal$. Since we have $f'\de\in\Ebb_{R}^{H}(B', C)$ by the argument in the previous paragraph, we have a right exact sequence 
  \[
  \begin{tikzcd}[column sep=1.5cm]
  HB' \rar["\bigl(\substack{Hb \\ -Hg'}\bigr)"] & HB\opl HC' \rar["(Hg Hc)"] & HC \rar & 0
  \end{tikzcd}
  \]
   in $\Acal$, hence $(*)$ is a push-out diagram. Since $Hg$ is an epimorphism, it follows that $Hg'$ is epimorphic in $\Acal$, hence we have $\de c\in\Ebb_{R}^{H}$. Thus $\Ebb_{R}^{H}$ is a subfunctor of $\Ebb$.

Next, we will show that $\Ebb_{R}^{H}(C, A)$ is a subgroup of $\Ebb(C, A)$. Clearly $0\in\Ebb_{R}^{H}(C, A)$. 
We only need to show $\de'-\de\in\Ebb_{R}^{H}(C, A)$ for any $\de, \de'\in\Ebb_{R}^{H}(C, A)$. 
Since we have $\de'-\de=(-\id \id)(\de\opl\de')(\substack{\id \\ \id})$ with $(-\id \id)\co A\opl A\to A$ and $(\substack{\id \\ \id})\co C\to C\opl C$, 
it is enough to show that $\de\opl\de'\in\Ebb_{R}^{H}(C\opl C, A\opl A)$, and this follows immediately from the definition of an additive realization $\sfra$. 
Thus $\Ebb_{R}^{H}$ is an additive subfunctor of $\Ebb$. 

Finally, it can be easily seen that $\sfra|_{\Ebb_{R}^{H}}$-deflations are closed under compositions. Thus $\Ebb_{R}^{H}$ is a closed subfunctor of $\Ebb$.

(2) and (3) follow immediately from the construction of $\Ebb_{R}^{H}$. 
\end{proof}

Dually, we obtain the statement with respect to $\Ebb_{L}^{H}$ by applying the above proposition to a contravariant half exact functor $H\co\Ccal\to\Acal^{\op}$.

The previous construction recovers Examples~\ref{ex:projectivize}, \ref{ex:proper}.

\begin{example}$($\cite[Proposition 1.7]{DRSSK}, \cite[Definition 3.18]{HLN}, \cite[Definition-Proposition 5.6]{INP}\label{ex:projectivize}$)$
Let $\Dcal$ be a subcategory of $\Ccal$. In \cite[Proposition 3.17]{HLN}, it is shown that additive subfunctors $\Ebb_{\Dcal}$ and $\Ebb^{\Dcal}$ of $\Ebb$ defined by
  \[
   \Ebb_{\Dcal}(C, A)=\{\de\in\Ebb(C, A)\mid(\de_{\sha})_{D}=0 \ \text{for any} \ D\in\Dcal\}
  \]
  \[
   \Ebb^{\Dcal}(C, A)=\{\de\in\Ebb(C, A)\mid(\de^{\sha})_{D}=0 \ \text{for any} \ D\in\Dcal\}
  \]
are closed subfunctors of $\Ebb$. They are the special cases of Proposition~\ref{prop:left} because the restricted Yoneda functors $Y_{\Dcal}\colon\Ccal\to\Mod\Dcal$ and $Y^{\Dcal}\colon\Ccal\to\Mod\Dcal^{\op}$ are half exact functors and $\Ebb_{\Dcal}=\Ebb_{R}^{Y_\Dcal}$ and $\Ebb^{\Dcal}=\Ebb_{R}^{Y^{\Dcal}}$ hold. Indeed, for any $\sfra$-triangle $A\ovs{f}{\to}B\ovs{g}{\to}C\ovs{\de}{\dra}$, we have an exact sequence
    \[
    \begin{tikzcd}
       \Ccal(-,B)|_{\Dcal} \rar["Y_{\Dcal}(g)"] & \Ccal(-, C)|_{\Dcal} \rar["\de_{\sha}"] & \Ebb(-, A)|_{\Dcal}
    \end{tikzcd}
    \]
in $\Mod\Dcal$, and the exactness implies $\Ebb_{R}^{Y_\Dcal}(C, A)=\Ebb_{\Dcal}(C, A)$. 
\end{example}

\begin{example}$($\cite[Example 2.3 (3),(4)]{Bel}\label{ex:proper}$)$
Let $\CEs$ be a triangulated category with a shift functor $\Si$ viewed as an extriangulated category and $H\co\Ccal\to\Acal$ a homological functor to an abelian category $\Acal$. 
We define $\Ecal$ as a class of distinguished triangles
  \begin{equation}\label{disttri}
    \begin{tikzcd}
      A \rar["f"] &B \rar["g"] &C\rar["\de"] & \Si A
    \end{tikzcd}
  \end{equation}
which satisfy that for any $i\in\Zbb$, the induced sequence
  \[
    \begin{tikzcd}
      0\rar & H^{i}A\rar["H^{i}f"] &H^{i}B\rar["H^{i}g"] &H^{i}C\rar & 0
    \end{tikzcd}
  \]
is exact in $\Acal$ with $H^{i}=H\ci\Si^i$. The class $\Ecal$ consists a \emph{proper class of triangles} in the sense of \cite{Bel}. Meanwhile the above condition is equivalent to $\de\in\bigcap_{i\in\Zbb}(\Ebb_{R}^{H^{i}}(C,A)\cap\Ebb_{L}^{H^{i}}(C,A))$, hence $\Ecal$ defines a relative theory of $\CEs$. 

Let $\Dcal$ be a subcategory of $\Ccal$ closed under $\Si$ and $\Si^{-1}$. We define $\Ecal(\Dcal)$ as a class of triangles $(\ref{disttri})$ which satisfy that the induced sequence
  \[
    \begin{tikzcd}
      0\rar & (-,A)|_{\Dcal}\rar["Y_{\Dcal}(f)"] &(-,B)|_{\Dcal}\rar["Y_{\Dcal}(g)"] &(-,C)|_{\Dcal}\rar & 0
    \end{tikzcd}
  \]
is exact in $\Mod\Dcal$. This condition is equivalent to the previous case of $H=Y_\Dcal$, therefore $\Ecal(\Dcal)$ is a proper class of triangles. Moreover if $\CEs$ is compactly generated and $\Dcal$ is the subcategory consisting of compact objects, then a distinguished triangle in $\Ecal(\Dcal)$ is called a \emph{pure triangle}.
\end{example}

\begin{example}
Every extriangulated category has a unique maximal exact relative theory, that is, there exists a unique maximal closed subfunctor $\Ebb^{\ex}$ of $\Ebb$ such that $(\Ccal, \Ebb^{\ex})$ is an exact category. 
Indeed, we set 
\[
 \Ebb^{\ex}:=\Ebb_{L}^{Y_\Ccal}\cap\Ebb_{L}^{Y^\Ccal}
\]
that is, any $\sfra|_{\Ebb^{\ex}}$-inflation is a monomorphic $\sfra$-inflation 
and any $\sfra|_{\Ebb^{\ex}}$-deflation is an epimorphic $\sfra$-deflation. 
By \cite[Corollary 3.18]{NP}, the relative theory $(\Ccal,\Ebb^{\ex})$ is an exact category.
\end{example}

Next we will observe relative theories which we have just constructed, in terms of functor categories. We define the following subcategories of $\Mod\Ccal$.
\[
 \defe\Ebb\se\coh\Ccal\se\mod\Ccal\se\Mod\Ccal
\]
See \cite[Subsection 2.2]{Eno} for more details.

\begin{definition}$($\cite[Definition 2.6]{Eno},\cite[Definition 2.4]{Oga}$)$
Let $M$ be a right $\Ccal$-module.
\begin{itemize}
\item $M$ is \emph{finitely generated} if there exists an epimorphism $\Ccal(-,C)\defl M$ for some $C\in\Ccal$. 
\item $M$ is \emph{finitely presented} if there exists an exact sequence
  \[
  \begin{tikzcd}
  \Ccal(-,B) \rar & \Ccal(-, C) \rar & M\rar & 0
  \end{tikzcd}
  \]
in $\Mod\Ccal$ with $B,C\in\Ccal$. We denote by $\mod\Ccal$ the subcategory of $\Mod\Ccal$ consisting of finitely presented $\Ccal$-modules.
\item $M$ is \emph{coherent} if $M$ is finitely presented and each finitely generated submodule of $M$ is finitely presented. We denote by $\coh\Ccal$ the subcategory of $\Mod\Ccal$ consisting of coherent $\Ccal$-modules.
\item Let $\de\in\Ebb(C, A)$ be an $\Ebb$-extension. We call the image of $\de_{\sharp}\co\Ccal(-,C)\to\Ebb(-,A)$ in $\Mod\Ccal$ the \emph{contravariant defect} of $\de$ and denote it by $\de^{\ast}$. By Proposition~\ref{prop:exactseq}, there is an exact sequence
    \begin{equation}\label{defect}
    \begin{tikzcd}
       \Ccal(-, A)\rar["f\ci-"]&\Ccal(-,B) \rar["g\ci-"] & \Ccal(-, C) \rar & \de^{\ast}\rar & 0
    \end{tikzcd}
    \end{equation}
for any $\sfra$-triangle $A\ovs{f}{\to}B\ovs{g}{\to}C\ovs{\de}{\dra}$. Dually the \emph{covariant defect} $\de_{\ast}\in\Mod\Ccal^{\op}$ of $\de$ is defined. We denote by $\defe\Ebb$ (resp. $\Ebb$-$\defe$) the subcategory of $\Mod\Ccal$ (resp. $\Mod\Ccal^{\op}$) consisting of $\Ccal$-modules isomorphic to contravariant (resp. covariant) defects of $\Ebb$-extensions. 
\end{itemize}
\end{definition}

\begin{remark}
\begin{enumerate}
\item It is well-known that $\coh\Ccal$ is an abelian category. We refer to \cite[Proposition 1.5]{Her} for the proof. If $\Ccal$ has weak kernels, then $\coh\Ccal=\mod\Ccal$ holds. 
\item If $\CEs$ is a triangulated category, then $\defe\Ebb=\mod\Ccal$ holds. This follows from that any morphism in $\Ccal$ is an $\sfra$-deflation.
\end{enumerate}
\end{remark}

In \cite{Eno}, closed subfunctors are classified by Serre subcategories of $\defe\Ebb$. In case $\Ccal$ has weak kernels, this result is generalized to $n$-exangulated categories by Hu-Ma-Zhang-Zhou in \cite{HMZZ}. 

\begin{theorem}$($\cite[Theorem A,B]{Eno}\label{thm:eno}$)$
Let $\CEs$ be an extriangulated category. 
\begin{enumerate}
  \item $\defe\Ebb$ is a Serre subcategory of $\coh\Ccal$, hence an abelian category.
  \item There exists a bijective correspondence between the set $\Rel(\Ebb)$ of closed subfunctors of $\Ebb$ and the set $\Serre(\defe\Ebb)$ of Serre subcategories of $\defe\Ebb$.
      \[
    \begin{tikzcd}[column sep=1cm]
      \Serre(\defe\Ebb) \rar["\Scal\mapsto\Fbb(\Scal)", shift left] & [2em] \Rel(\Ebb) \lar["\Fbb\mapsto\defe\Fbb", shift left] 
    \end{tikzcd}
    \]
Here, $\Fbb(\Scal)$ is defined by $\Fbb(\Scal)(C,A)=\{\de\in\Ebb(C,A)\mid\de^{\ast}\in\Scal\}$ and $\defe\Fbb$ is defined by $\defe\Fbb=\{F\in\defe\Ebb\mid F\iso\de^{\ast} \ \text{for some} \ \de\in\Fbb(C,A)\}$.     
\end{enumerate}
\end{theorem}

In Proposition~\ref{prop:serel}, we observe a relationship between the above theorem and closed subfunctors constructed in Proposition~\ref{prop:left}. The following proposition is useful to associate half exact functors to functor categories.

\begin{proposition}$($\cite[Universal Property 2.1]{Kra}\label{prop:kra}$)$
Let $\Ccal$ be an additive category and $H\colon\Ccal\to\Acal$ an additive functor to an additive category $\Acal$ which has cokernels. 
Then there exists a right exact functor $\wti{H}\colon\mod\Ccal\to\Acal$ which makes the following diagram commute up to a natural isomorphism, and such a functor is unique up to a natural isomorphism.
  \[
  \begin{tikzcd}
    \Ccal \rar["Y_{\Ccal}"] \ar[rd, "H"'] & \mod\Ccal \dar["\wti{H}"] \\
    \ & \Acal
  \end{tikzcd}
  \]
Here $Y_{\Ccal}$ denotes the Yoneda embedding $\Ccal\to\mod\Ccal; C\mapsto \Ccal(-,C)$.
\end{proposition}

In the above, the functor $\wti{H}$ is determined uniquely up to a natural isomorphism by the property that $\wti{H}$ sends a right exact sequence $\Ccal(-,B)\ovs{g\ci-}{\to}\Ccal(-,C)\to F\to 0$ in $\mod\Ccal$ to a right exact sequence $HB\ovs{Hg}{\to}HC\to\wti{H}F\to0$ in $\Acal$.
From the previous proposition, it turns out that any half exact functor induces an exact functor whose domain is $\defe\Ebb$:

\begin{proposition}\label{prop:liftexact}
Let $H\colon \Ccal\to\Acal$ be a half exact functor. 
Then the induced functor $\wti{H}$ restricts to an exact functor $\wha{H}\colon\defe\Ebb\to\Acal$. 
\end{proposition}

\begin{proof}
By Proposition \ref{prop:kra}, we know that $\wha{H}$ is a right exact functor. 
Consider a short exact sequence $0\to\de^{\ast}\to\de''^{\ast}\to\de'^{\ast}\to 0$ in $\defe\Ebb$, 
and let $A\ovs{f}{\to} B\ovs{g}{\to} C\ovs{\de}{\dra}$ and $A'\ovs{f'}{\to} B'\ovs{g'}{\to} C'\ovs{\de'}{\dra}$ be $\sfra$-triangles. 
By the horseshoe lemma, we obtain the following exact commutative diagram in $\Mod\Ccal$.
  \[
  \begin{tikzcd}
    0\rar & \Ccal(-,A) \rar["\bigl(\substack{1 \\ 0}\bigr)"] \dar["f\ci-"] & \Ccal(-,A)\opl \Ccal(-,A') \rar["(01)"] \dar & \Ccal(-,A') \rar\dar["f'\ci-"] & 0  \\  
    0\rar & \Ccal(-,B) \rar["\bigl(\substack{1 \\ 0}\bigr)"] \dar["g\ci-"] & \Ccal(-,B)\opl \Ccal(-,B') \rar["(01)"] \dar & \Ccal(-,B') \rar\dar["g'\ci-"] & 0  \\
    0\rar & \Ccal(-,C) \rar["\bigl(\substack{1 \\ 0}\bigr)"] \dar & \Ccal(-,C)\opl \Ccal(-,C') \rar["(01)"] \dar & \Ccal(-,C') \rar\dar & 0  \\
    0\rar & \de^{\ast} \rar\dar & \de''^{\ast} \rar\dar & \de'^{\ast} \rar\dar & 0 \\
    \ & 0 & 0 & 0& \ 
  \end{tikzcd}
  \]
Applying $\wti{H}$ to the above, we obtain the following commutative diagram in $\Acal$. 
  \[
  \begin{tikzcd}
    0\rar & HA \rar["\bigl(\substack{1 \\ 0}\bigr)"] \dar["Hf"] & HA\opl HA' \rar["(01)"] \dar & HA' \rar\dar["Hf'"] & 0  \\  
    0\rar & HB \rar["\bigl(\substack{1 \\ 0}\bigr)"] \dar["Hg"] & HB\opl HB' \rar["(01)"] \dar & HB' \rar\dar["Hg'"] & 0  \\
    0\rar & HC \rar["\bigl(\substack{1 \\ 0}\bigr)"] \dar & HC\opl HC' \rar["(01)"] \dar & HC' \rar\dar & 0  \\
       \   & \wha{H}\de^{\ast} \rar["\varphi"]\dar & \wha{H}\de''^{\ast} \rar\dar & \wha{H}\de'^{\ast} \rar\dar & 0 \\
       \ & 0 & 0 & 0& \     
  \end{tikzcd}
  \]
Here all rows are exact, and left and right columns are exact, while the sequence $HA\oplus HA'\to HB\oplus HB'\to HC\oplus HC'$ is not necessary exact. Then it follows from the snake lemma and the exactness of the sequence $HA'\ovs{Hf'}{\to}HB'\ovs{Hg'}{\to}HC'$ that $\varphi$ is monomorphic. 
Thus $\wha{H}$ is an exact functor. 
\end{proof}

By a dual argument, we obtain the following claim:
Let $H\co\Ccal\to\Acal$ be a half exact functor. Then there exists a contravariant left exact functor $\wti{H}\co\mod\Ccal^{\op}\to\Acal$ satisfying $H\iso\wti{H}\ci Y^{\Ccal}$ where $Y^{\Ccal}$ denotes the Yoneda embedding $\Ccal\to\mod\Ccal^{\op}; C\mapsto\Ccal(C,-)$, and $\wti{H}$ restricts to a contravariant exact functor $\wch{H}\co\Ebb\mathchar`-\defe\to\Acal$. 

We obtain a duality between the category of contravariant defects and that of covariant defects. 

\begin{proposition}
There is a contravariant exact equivalence
    \[
    \begin{tikzcd}[column sep=0.5cm]
      D\co\defe\Ebb \rar["\sim"] & \Ebb\mathchar`-\defe 
    \end{tikzcd}
    \]
    which sends $\de^{\ast}$ to $\de_{\ast}$ for any $\Ebb$-extension $\de$.
\end{proposition}
\begin{proof}
Consider a functor $F$ defined by
  \[
      \begin{tikzcd}
       \Ccal \rar & \Mod\Ccal^{\op} \ ; \ X \rar[mapsto] & \Ebb(X,-)
    \end{tikzcd}
  \]
which is a contravariant half exact functor.
By Proposition~\ref{prop:liftexact}, we obtain a contravariant exact functor $D=\wha{F}\co\defe\Ebb\to\Mod\Ccal^{\op}$. For any $\Ebb$-extension $\de$, there exists an exact sequence as (\ref{defect}). By Proposition~\ref{prop:exactseq}, we have an exact sequence
  \[
    \begin{tikzcd}
       \Ccal(A,-) \rar["\de^{\sha}"] & \Ebb(C,-) \rar["g^{\ast}"]  & \Ebb(B,-).
    \end{tikzcd}
  \]
Then we have $D\de^{\ast}=\wti{F}(\Cokernel(g\ci-))=\Kernel g^{\ast}=\Image \de^{\sha}=\de_{\ast}$, and get $D\co\defe\Ebb\to\Ebb\mathchar`-\defe$. Dually, we consider a functor $G$ defined by
  \[
      \begin{tikzcd}
       \Ccal \rar & \Mod\Ccal \ ; \ X \rar[mapsto] & \Ebb(-,X).
    \end{tikzcd}
  \]
Dually, we obtain a contravariant exact functor $\wch{G}\co\Ebb\mathchar`-\defe\to\defe\Ebb$, and it is easy to check that $D$ and $\wch{G}$ are mutually quasi-inverses.  
\end{proof}

The following proposition gives other descriptions of Serre subcategories corresponding to $\Ebb^{H}_{R}$ and $\Ebb^{H}_{L}$ via the bijection in Theorem~\ref{thm:eno}.

\begin{proposition}\label{prop:serel}
Let $H\colon\Ccal\to\Acal$ be a half exact functor and $\de\in\Ebb(C, A)$ any $\Ebb$-extension. 
Then the following assertions hold. 
\begin{enumerate}
  \item $\de\in\Ebb^{H}_{R}(C, A)$ if and only if $\de^{\ast}\in\ker\wha{H}$. Hence we have $\ker\wha{H}=\defe\Ebb^{H}_{R}$. 
  \item $\de\in\Ebb^{H}_{L}(C, A)$ if and only if $\de^{\ast}\in\ker(\wch{H}\ci D)$. Hence we have $\ker(\wch{H}\ci D)=\defe\Ebb^{H}_{L}$.
\end{enumerate}
\end{proposition}

\begin{proof}
Let $A\ovs{f}{\to}B\ovs{g}{\to}C\ovs{\de}{\dra}$ be an $\sfra$-triangle. 

(1) follows from the exactness of the sequence $HB\ovs{Hg}{\to}HC\to\wha{H}\de^{\ast}\to 0$. 
Similarly, (2) follows from the sequence $0\to(\wch{H}\ci D)(\de^{\ast})\to HA\ovs{Hf}{\to}HB$.
\end{proof}

Proposition~\ref{prop:serel} implies that the following diagram is commutative
    \[
    \begin{tikzcd}[column sep=2cm]
      \Hex(\Ccal) \ar[rd,"H\mapsto\Ebb_{R}^{H}",yshift=0.5ex] \dar["H\mapsto\ker\wha{H}"'] & \\
      \Serre(\defe\Ebb) \rar["\Scal\mapsto\Fbb(\Scal)", shift left] & [2em] \Rel(\Ebb) \lar["\Fbb\mapsto\defe\Fbb", shift left] 
    \end{tikzcd}
    \] 
where $\Hex(\Ccal)$ denotes the class of all half exact functors from $\Ccal$ to abelian categories. 

Next we discuss when two half exact functors define the same closed subfunctor. The following definition is an analogue of \cite[Definition 6.1]{SS}.

\begin{definition}
Let $H\co\Ccal\to\Acal$ and $H'\co\Ccal\to\Acal'$ be half exact functors. We say $H'$ is a \emph{faithfully exact reduction} of $H$ if there exists a faithful exact functor $F\co\Acal\to\Acal'$ satisfying $H'\iso F\ci H$. 
\end{definition}

\begin{proposition}\label{prop:same}
Let $H$ and $H'$ be half exact functors. If $H'$ is a faithfully exact reduction of $H$, then $\Ebb_R^{H}=\Ebb_R^{H'}$ holds.
\end{proposition}
\begin{proof}
Assume there exists a faithful exact functor $F\co\Acal\to\Acal'$ such that $F\ci H\iso H'$. By Proposition~\ref{prop:kra}, we have $\wti{H'}\iso F\ci\wti{H}$. Since $F$ is a faithful functor, we obtain $\ker\wti{H'}=\ker\wti{H}$. Hence $\ker\wha{H'}=\ker\wti{H'}\cap\defe\Ebb=\ker\wti{H}\cap\defe\Ebb=\ker\wha{H}$ holds. By Proposition~\ref{prop:serel}, we obtain the desired result.
\end{proof}

We end this section by giving the classification of closed subfunctors of $\CEs$ via projectivization functors. The following definition is required for the definition of projectivization functors (Definition~\ref{def:projectivize}).

\begin{definition}$($\cite[Definitions 3.23, 3.25]{NP}$)$
An object $P\in\Ccal$ is \emph{projective} if we have $\Ebb(P, C)=0$ for any $C\in\Ccal$. Denote by $\Pcal$ the subcategory consisting of all projective objects. 
An extriangulated category $\CEs$ \emph{has enough projective objects} if for any $C\in\Ccal$, there is an $\sfra$-triangle $A\to P\to C\dra$ such that $P$ is projective. We denote by $\un{\Ccal}$ the ideal quotient of $\Ccal$ by the ideal consisting of all morphisms which factor through projective objects.
\end{definition}

Note that all triangulated categories have enough projective objects and the subcategory $\Pcal$ is zero.
To give the classification result, we assume that $\CEs$ has enough projective objects in the rest of this section.
Thanks to the assumption, we obtain the following proposition.

\begin{proposition}\label{prop:gnp}$($\cite[Proposition 6.5]{GNP}$)$
If $\CEs$ has enough projective objects, then the inclusion $\iota\co\defe\Ebb\to\mod\Ccal$ has a left adjoint functor. 
In this case, the left adjoint functor $\Ga\co\mod\Ccal\to\defe\Ebb$ of $\iota$ sends $\Ccal(-,X)$ to $\un{\Ccal}(-,X)$.
\end{proposition}

Note that $\defe\Ebb=\mod\Ccal$ and $\Ga=\Id$ hold when $\CEs$ is a triangulated category. The following proposition is an analogous result of Proposition~\ref{prop:kra}.

\begin{proposition}\label{prop:defhex}
\begin{enumerate}
\item The functor $\Ga\ci Y_{\Ccal}\co\Ccal\to\defe\Ebb$ is a half exact functor which sends all objects in $\Pcal$ to zero.
\item For any half exact functor $H\co\Ccal\to\Acal$ which satisfies $H(P)=0$ for each $P\in\Pcal$, there exists a unique exact functor $\wha{H}\co\defe\Ebb\to\Acal$ which makes the following diagram commute up to a natural isomorphism, and such a functor is unique up to a natural isomorphism.
  \[
  \begin{tikzcd}
    \Ccal \rar["\Ga\ci Y_{\Ccal}"] \ar[rd, "H"'] & \defe\Ebb \dar["\wha{H}"] \\
    \ & \Acal
  \end{tikzcd}
  \]
\end{enumerate}  
\end{proposition}

\begin{proof}
(1) This follows immediately from  \cite[Lemma 1.27]{INP}.

(2) Let $(\iota,\Ga)$ be the adjoint pair obtained in Proposition~\ref{prop:gnp} and $\eta\co\Id\Rightarrow\iota\ci\Ga$ the unit morphism. Then we have $\wha{H}=\wti{H}\ci\iota$. We show that $\varphi=\wti{H}\ci\eta\ci Y_\Ccal\co\wti{H}\ci Y_\Ccal\Rightarrow\wti{H}\ci\iota\ci\Ga\ci Y_\Ccal$ is an isomorphism. For any $C\in\Ccal$, there exists an $\sfra$-triangle $A\to P\ovs{g}{\to}C\dra$ with a projective object $P$. Then we obtain the right exact sequence
    \[
    \begin{tikzcd}
       \Ccal(-,P) \rar["g\ci-"] & \Ccal(-, C) \rar["(\eta\ci Y_\Ccal)_C"] & \un{\Ccal}(-,C)\rar & 0
    \end{tikzcd}
    \]
in $\mod\Ccal$. Applying $\wti{H}$ to this sequence, we obtain the exact sequence
    \[
    \begin{tikzcd}
       HP \rar["Hg"] & HC \rar["\varphi_C"] & \wti{H}(\un{\Ccal}(-,C))\rar & 0
    \end{tikzcd}
    \]
in $\Acal$. By the assumption, we have $HP=0$. Hence $\varphi_C$ is an isomorphism. 

We show the uniqueness of such $\wha{H}$. Let $G\co\defe\Ebb\to\Acal$ be an exact functor satisfying $H\iso G\ci\Ga\ci Y_\Ccal$. Since $\Ga$ is a right exact functor, we obtain $\wti{H}\iso G\ci\Ga$ by the uniqueness in Proposition~\ref{prop:kra}. Thus we have $\wha{H}=\wti{H}\ci\iota\iso G\ci\Ga\ci\iota\iso G$.
\end{proof}

The following notion is considered in \cite[Definition 3.3]{Bel} for triangulated categories, which we adapt to extriangulated categories.

\begin{definition}\label{def:projectivize}
A half exact functor $H$ is called a \emph{projectivization functor} if $H=Q\ci\Ga\ci Y_\Ccal$ holds where $Q\co\defe\Ebb\to(\defe\Ebb)/\Scal$ is a Serre quotient functor for some Serre subcategory $\Scal$ of $\defe\Ebb$.
\end{definition}

\begin{proposition}\label{prop:fer}
Let $H\colon\Ccal\to\Acal$ be a half exact functor. Assume that $H(P)=0$ holds for any $P\in\Pcal$. Then $H$ is a faithfully exact reduction of a projectivization functor. 
\end{proposition}

\begin{proof}
By Proposition~\ref{prop:defhex}, we obtain an exact functor $\wha{H}\co\defe\Ebb\to\Acal$ such that $H\iso\wha{H}\ci\Ga\ci Y_\Ccal$. By \cite[Lemma 3.2]{SS}, there exists a factorization $\wha{H}=J\ci Q$, where $Q$ is a Serre quotient functor and $J$ is a faithful exact functor. Thus we have $H\iso J\ci(Q\ci\Ga\ci Y_\Ccal)$ as desired. 
\end{proof}

Combining Proposition~\ref{prop:same} and the previous proposition, it turns out that every half exact functor which sends all projective objects to zero defines a closed subfunctor arising from a certain projectivization functor.

The following is our main result.

\begin{theorem}\label{thm:tri}
Let $\CEs$ be an extriangulated category which has enough projective objects and $\Pcal$ the subcategory consisting of projective objects.
Then there exist bijective correspondences between
\begin{enumerate}
  \item projectivization functors
  \item Serre subcategories of $\defe\Ebb$
  \item closed subfunctors of $\Ebb$
\end{enumerate}
as in the following commutative diagram
    \[
    \begin{tikzcd}[column sep=2cm]
      \PFun(\Ccal)\ar[rd,"H\mapsto\Ebb_{R}^{H}",yshift=0.5ex] \dar["H\mapsto\ker\wha{H}"'] & \\
      \Serre(\defe\Ebb) \rar["\Scal\mapsto\Fbb(\Scal)", shift left] & [2em] \Rel(\Ebb) \lar["\Fbb\mapsto\defe\Fbb", shift left]
    \end{tikzcd}
    \] 
where $\PFun(\Ccal)$ denotes the set of projectivization functors from $\Ccal$.
\end{theorem}

\begin{proof}
The bijection between $(2)$ and $(3)$ is just Theorem~\ref{thm:eno}, and the one between $(1)$ and $(2)$ follows from Proposition~\ref{prop:defhex}. 
The commutativity of the diagram follows from Proposition~\ref{prop:serel}. 
\end{proof}

\begin{corollary}\label{cor:surj}
Let $\CEs$ be an extriangulated category which has enough projective objects. Then every closed subfunctor of $\Ebb$ is of the form $\Ebb_{R}^{H}$ for a certain half exact functor $H$. 
\end{corollary}

\section*{Appendix A. Closed subfunctors and proper classes}\label{app}
\setcounter{section}{1}
\renewcommand{\thesection}{\Alph{section}}
\setcounter{theorem}{0}

In this appendix, we show that the notions of closed subfunctors and proper classes are essentially the same. This is dealt in \cite{HMZZ} for $n$-exangulated categories. For the convenience of the reader, we give the proof in extriangulated categories. Throughout this section, let $\CEs$ be an extriangulated category. The following proposition is required for the definition of proper classes.

\begin{proposition}$($\cite[Proposition 3.15]{NP}\label{prop:shifted}$)$
Let $A_i\ovs{x_i}{\to}B_i\ovs{y_i}{\to}C\ovs{\de_i}{\dra}$ be $\sfra$-triangles for $i=1,2$. Then there exists a commutative diagram
  \begin{equation}\label{shifted}
  \begin{tikzcd}
     \ & A_2 \rar[equal] \dar["m_2"'] & A_2 \dar["x_2"] & \  \\
    A_1 \rar["m_1"]\dar[equal] & M \rar["e_1"]\dar["e_2"] & B_2 \rar[dashed, "\de_{1}y_2"]\dar["y_2"] & \ \\
     A_1\rar["x_1"] & B_1\rar["y_1"]\dar[dashed,"\de_{2}y_1"] & C\dar[dashed,"\de_2"]\rar[dashed, "\de_1"] & \ \\
     \ & \ & \ & \ 
  \end{tikzcd}
  \end{equation}
such that the middle row and the middle column are $\sfra$-triangles and $m_1\de_1+m_2\de_2=0$ holds.  
\end{proposition}

\begin{definition}\label{def:proper}
Let $\xi$ be a class of $\sfra$-triangles which is closed under isomorphisms. 
We call $\xi$ a \emph{proper class of $\sfra$-triangles} if it satisfies
\begin{enumerate}
  \item $\xi$ contains all split $\sfra$-triangles, that is, $\sfra$-triangles $A\ovs{f}{\to}B\ovs{g}{\to}C\ovs{\de}{\dra}$ such that $\de=0$. 
  \item $\xi$ is closed under direct sums, that is, for any $\sfra$-triangles $A\ovs{f}{\to}B\ovs{g}{\to}C\ovs{\de}{\dra}$, $A'\ovs{f'}{\to}B'\ovs{g'}{\to}C'\ovs{\de'}{\dra}$ in $\xi$, we have $A\opl A'\ovs{f\opl f'}{\to}B\opl B'\ovs{g\opl g'}{\to}C\opl C'\ovs{\de\opl\de'}{\dra}\in\xi$.
  \item $\xi$ is closed under base change, that is, for any $\sfra$-triangle $A\ovs{f}{\to}B\ovs{g}{\to}C\ovs{\de}{\dra}$ in $\xi$ and any morphism $c\co C'\to C$, we have $A\ovs{f'}{\to}B'\ovs{g'}{\to}C'\ovs{\de c}{\dra}\in\xi$.
  \item $\xi$ is closed under cobase change, dual of {\rm (3)}.
\end{enumerate}
Moreover $\xi$ is \emph{saturated} if it satisfies the following condition: In the diagram (\ref{shifted}), if the third column and the middle row belong to $\xi$, then the third row also belongs to $\xi$. 
\end{definition}

\begin{remark}
In \cite{HZZ}, the term of a class of $\Ebb$-triangles always means saturated one. However, we consider proper classes  without the saturatedness because it has to do with the closedness of subfunctors of $\Ebb$. See the following proposition.
\end{remark}

\begin{proposition}
Let $\CEs$ be an extriangulated category. 
\begin{enumerate}
  \item There exists a bijective correspondence between the class of additive subfunctors of $\Ebb$ and the class of proper classes of $\sfra$-triangles. 
  \item In the above correspondence, closed subfunctors of $\Ebb$ correspond to saturated proper classes of $\sfra$-triangles.
\end{enumerate}  
\end{proposition}

\begin{proof}
{\rm (1)} Let $\xi$ be a proper class of $\sfra$-triangles. We define
\[
 \Ebb_{\xi}(C,A):=\{\de\in\Ebb(C,A)\mid A\ovs{f}{\to}B\ovs{g}{\to}C\ovs{\de}{\dra}\in\xi\}
\]
for each $A,C\in\Ccal$, and show that $\Ebb_{\xi}$ is an additive subfunctor of $\Ebb$. Obviously, $\Ebb_{\xi}$ is a subfunctor of $\Ebb$ by Definition~\ref{def:proper} {\rm (3)} and {\rm (4)}. Hence we need to show that $\Ebb_{\xi}(C,A)\se\Ebb(C,A)$ is a subgroup for any $A,C\in\Ccal$. By Definition~\ref{def:proper} {\rm (1)}, zero element $0\in\Ebb(C,A)$ is contained in $\Ebb_{\xi}(C,A)$. Let $\de_1$ and $\de_2$ be any elements in $\Ebb_{\xi}(C,A)$. Since $\de_2-\de_1=(-\id \id)(\de_1\opl\de_2)(\sst{\id \\ \id})$, we have $\de_2-\de_1\in\Ebb_{\xi}(C,A)$ by Definition~\ref{def:proper} {\rm (2), (3)} and {\rm (4)}. Thus $\Ebb_{\xi}(C,A)$ is a subgroup of $\Ebb(C,A)$, hence $\Ebb_{\xi}$ is an additive subfunctor of $\Ebb$.

Conversely, let $\Fbb$ be an additive subfunctor of $\Ebb$. Consider a class $\xi_{\Fbb}$ of $\sfra$-triangles $A\ovs{f}{\to}B\ovs{g}{\to}C\ovs{\de}{\dra}$ satisfying $\de\in\Fbb(C,A)$. We will show that $\xi_{\Fbb}$ is a proper class of $\sfra$-triangles.  First, $\xi_{\Fbb}$ satisfies Definition~\ref{def:proper} {\rm (1)} and {\rm (2)} because $\Fbb$ is additive. Next, $\xi_{\Fbb}$ is closed under isomorphisms and satisfies Definition~\ref{def:proper} {\rm (3)} and {\rm (4)} because $\Fbb$ is a subfunctor of $\Ebb$. Thus $\xi_{\Fbb}$ is a proper class of $\sfra$-triangles. 

It is easy to see that the above correspondence is bijective. 

{\rm (2)} Let $\Fbb$ be a closed subfunctor of $\Ebb$. Then it follows that $\xi_{\Fbb}$ is saturated from \cite[Lemma 3.14]{HLN}. Conversely, let $\xi$ be a saturated proper class of $\sfra$-triangles. We show that $\Ebb_{\xi}$ is a closed subfunctor of $\Ebb$. By Proposition~\ref{prop:closed}, it is enough to show that $\sfra_{\Ebb_{\xi}}$-deflations are closed under compositions. Let $A\ovs{f}{\to}B\ovs{g}{\to}C\ovs{\de}{\dra}$ and $D\ovs{f'}{\to}C\ovs{g'}{\to}E\ovs{\de'}{\dra}$ be $\sfra_{\Ebb_{\xi}}$-triangles. By {\rm (ET4)}$^{\op}$ in $\CEs$, we obtain the following commutative diagram. 
  \[
  \begin{tikzcd}
    A \rar \dar[equal] & F \rar \dar & D \rar[dashed, "\de f'"]\dar["f'"] & \  \\
    A \rar["f"] & B \rar["g"]\dar["g'g"'] & C \rar[dashed, "\de"]\dar["g'"] & \ \\
     \ & E\rar[equal]\dar[dashed,"\rho"] & E\dar[dashed,"\de'"] & \ \\
     \ & \ & \ & \ 
  \end{tikzcd}
  \]
We need to show that $\rho\in\Ebb_{\xi}(E,F)$. By Proposition~\ref{prop:shifted}, we obtain the following commutative diagram.
  \begin{equation}\label{satur}
  \begin{tikzcd}
     \ & D \rar[equal] \dar & D \dar["f'"] & \  \\
    F \rar\dar[equal] & M \rar["e"]\dar & C \rar[dashed, "\rho g'"]\dar["g'"] & \ \\
     F\rar & B\rar["g'g"']\dar[dashed, "0"'] & E\dar[dashed,"\de'"]\rar[dashed, "\rho"] & \ \\
     \ & \ & \ & \ 
  \end{tikzcd}
  \end{equation}
Then there exists a morphism $s\co B\to M$ such that $es=g$ by Proposition~\ref{prop:exactseq} $(1)$. Applying {\rm (ET3)}$^{\op}$ to
  \[
  \begin{tikzcd}
    A \rar["f"] & B \rar["g"] \dar["s"] & C \rar[dashed, "\de"]\dar[equal] & \  \\
    F \rar & M \rar["e"] & C \rar[dashed, "\rho g'"] & \ ,
  \end{tikzcd}
  \]
there exists a morphism $a\co A\to F$ such that $(a, s, \id_C)$ gives a morphism of $\sfra$-triangles. Then we obtain $F\to M\ovs{e}{\to}C\ovs{\rho g'}{\dra}\in\xi$ by Definition~\ref{def:proper} {\rm (4)}. Since $\xi$ is saturated, we have $F\to B\ovs{g'g}{\to}E\ovs{\rho}{\dra}\in\xi$ by the diagram (\ref{satur}).
\end{proof}

\begin{ack}
  The author would like to thank his supervisor Hiroyuki Nakaoka for helpful discussions.
\end{ack}

\end{document}